\documentclass[11pt, notitlepage]{article}
\usepackage{amssymb,amsmath,comment}
\usepackage{amsthm}
\catcode`\@=11 \@addtoreset{equation}{section}
\def\thesection{\arabic{section}}

\def\theequation{\thesection.\arabic{equation}}
\catcode`\@=12
\usepackage{colortbl}
\usepackage{hyperref}
\usepackage[mathscr]{eucal}
\usepackage{epsf}
\usepackage{esint}
\usepackage{a4wide}

\newcommand{\Om} {\Omega}

\newcommand{\la} {\lambda}
\newcommand{\La} {\Lambda}
\newcommand{\noi} {\noindent}

\newcommand{\uline} {\underline}
\newcommand{\oline} {\overline}
\newcommand{\mb} {\mathbb}
\newcommand{\mc} {\mathcal}

\usepackage[all]{xy}
\catcode`\@=11
\def\theequation{\@arabic{\c@section}.\@arabic{\c@equation}}
\catcode`\@=12

\newtheorem{Theorem}{Theorem}[section]
\newtheorem{Lemma}[Theorem]{Lemma}

\newtheorem{Remark}[Theorem]{Remark}
\newtheorem{Definition}[Theorem]{Definition}

\begin{document}

{\vspace{0.01in}}

\title{On a class of mixed local and nonlocal semilinear elliptic equation with singular nonlinearity}
\author{Prashanta Garain}

\date{}
\maketitle

\begin{abstract}
In this article, we consider a combination of local and nonlocal Laplace equation with singular nonlinearities. For such mixed problems, we establish existence of at least one weak solution for a parameter dependent singular nonlinearity and existence of multiple solutions for purturbed singular nonlinearity. Our argument is based on the variational and approximation approach. 
\end{abstract}

\maketitle

\noi {Keywords: Mixed local and nonlocal equation, singular nonlinearity, existence, regularity.}

\noi{\textit{2020 Mathematics Subject Classification: 35M10, 35R11, 35B65, 35J75.}

\tableofcontents
\maketitle

\section{Introduction}
In this article, we consider the following mixed local and nonlocal semilinear equation with singular nonlinearity
\begin{equation}\label{maineqn}
-\Delta u+(-\Delta)^s u=g(x,u)\text{ in }\Omega,\quad u>0\text{ in }\Omega,\quad u=0\text{ in }\mathbb{R}^n\setminus\Omega,
\end{equation}
where $\Omega\subset\mathbb{R}^n$ is a bounded domain with $n\geq 2$. Here $-\Delta$ is the classical Laplace operator
and $(-\Delta)^s$, $s\in(0,1)$ is the fractional Laplace operator defined by
$$
(-\Delta)^s u=\text{P.V.}\int_{\mathbb{R}^n}\frac{u(x)-u(y)}{|x-y|^{n+2s}}\,dy,
$$
where P.V. denotes the principal value.
We establish existence of at least one weak solution of the problem \eqref{maineqn} for the purely singular nonlinearity $g$ of the form $(g_1)$ given by
$$
g(x,u)=\la h(u)u^{-\gamma},\quad \leqno(g_1)
$$
where $\la>0,\gamma\in(0,1)$ and 
\begin{enumerate}
\item[$(h_1)$] $h:[0,\infty)\to\mathbb{R}$ is a continuous nondecreasing function such that $h(0)>0$ and
\item[$(h_2)$]
\begin{equation}\label{H1f1}
\lim_{t\to 0}\frac{h(t)}{t^\gamma}=\infty, \quad \lim_{t\to\infty}\frac{h(t)}{t^{\gamma+1}}=0.
\end{equation}
\end{enumerate}
Further, we establish multiplicity result for the equation \eqref{maineqn} with the purturbed singular nonlinearity $g$ of the form $(g_2)$ given by
$$
g(x,u)=\la u^{-\gamma}+u^q,\quad\leqno(g_2)
$$
where $\la>0$, $\gamma\in(0,1)$ and $q\in(1,2^*-1)$ with $2^*=\frac{2n}{n-2}$ if $n>2$ and $2^*=\infty$ if $n=2$.

Before proceeding further, we state the functional setting to study the problem \eqref{maineqn}.

\subsection{Functional setting and useful results}
In this section, we present some known results for the fractional Sobolev space, see \cite{Hitchhiker'sguide} for more details.
Let 
$E\subset \mathbb{R}^n$
be a measurable set and $|E|$ denote its Lebesgue measure. Recall that the Lebesgue space 
$L^{2}(E),$
is defined as the space of measurable functions $u:E\to\mathbb{R}$ with the finite norm
$$ \|u\|_{L^2(E)}=
\left(\int\limits_{E}|u(x)|^2~dx\right)^{1/2}.
$$
Here and in the rest of the paper, it is assumed that $\Om\subset\mathbb{R}^n$ with $n\geq 2$ is a bounded smooth domain. The Sobolev space $H^1(\Omega)$ is defined 
as the Banach space of locally integrable weakly differentiable functions
$u:\Omega\to\mathbb{R}$ equipped with the following norm: 
\[
\|u\|_{H^{1}(\Omega)}=\| u\|_{L^2(\Omega)}+\|\nabla u\|_{L^2(\Omega)}.
\]
The space $H^{1}(\mathbb{R}^n)$ is defined analogously.
To deal with mixed problems, we use the space $H^{1}_0(\Omega)=\{u\in H^{1}(\mathbb{R}^n):u=0\text{ in }\mathbb{R}^n\setminus\Om\}$ under the norm $\|u\|=\|\nabla u\|_{L^2(\Om)}$. It can be shown that $H^{1}_0(\Omega)$ is a real separable and reflexive Banach space, see \cite{BDVV2, Biagi1, SV}.  

The fractional Sobolev space $H^{s}(\Omega)$, $0<s<1$, is defined by
$$
H^{s}(\Omega)=\Big\{u\in L^2(\Omega):\frac{|u(x)-u(y)|}{|x-y|^{\frac{n}{2}+s}}\in L^2(\Omega\times \Omega)\Big\},
$$
which is endowed with the norm
$$
\|u\|_{H^{s}(\Omega)}=\left(\int_{\Omega}|u(x)|^2\,dx+\int_{\Omega}\int_{\Omega}\frac{|u(x)-u(y)|^2}{|x-y|^{n+2s}}\,dx\,dy\right)^\frac{1}{2}.
$$
For the next result, see \cite[Proposition $2.2$]{Hitchhiker'sguide}.
\begin{Lemma}\label{defineq}
There exists a constant $C=C(n,s)>0$ such that
$$
\|u\|_{H^{s}(\Om)}\leq C\|u\|_{H^{1}(\Om)},\quad\forall\,u\in H^{1}(\Omega).
$$
\end{Lemma}

Next, we have the following result from \cite[Lemma $2.1$]{Silva}.
\begin{Lemma}\label{locnon1}
There exists a constant $C=C(n,s,\Omega)$ such that
\begin{equation}\label{locnonsem}
\iint\limits_{\mathbb{R}^{2n}}\frac{|u(x)-u(y)|^2}{|x-y|^{n+2s}}\,dx\,dy\leq C\int_{\Omega}|\nabla u|^2\,dx,\quad\forall\,u\in H_0^{1}(\Omega).
\end{equation}
\end{Lemma}

For the following Sobolev embedding, see, for example, \cite{Evans}.

\begin{Lemma}\label{emb}
The embedding operators
\[
H_0^{1}(\Omega)\hookrightarrow
\begin{cases}
L^t(\Om),&\text{ for }t\in[1,2^{*}],\text{ if }n>2,\\
L^t(\Om),&\text{ for }t\in[1,\infty),\text{ if }n=2
\end{cases}
\]
are continuous.
\end{Lemma}

Now we are ready to define the notion of weak solutions for the problem \eqref{maineqn}. 

\begin{Definition}\label{wksoldef}(Weak Solution)
Let $g$ be either of the form $(g_1)$ or $(g_2)$. We say that $u\in H_0^{1}(\Omega)$ is a weak subsolution (or supersolution) of \eqref{maineqn}, if $u>0$ in $\Omega$ such that for every $\omega\Subset\Omega$, there exists a positive constant $c(\omega)$ with $u\geq c(\omega)>0$ in $\omega$ and
\begin{equation}\label{wksoleqn}
\int_{\Omega}\nabla u\nabla\phi\,dx+\iint_{\mathbb{R}^{2n}}\frac{(u(x)-u(y))(\phi(x)-\phi(y))}{|x-y|^{n+2s}}\,dx dy\leq (\text{ or })\geq \int_{\Omega}g(x,u)\phi\,dx,
\end{equation}
for every nonnegative $\phi\in C_c^{1}(\Omega)$. We say that $u\in H_0^{1}(\Omega)$ is a weak solution of \eqref{maineqn}, if the equality in \eqref{wksoleqn} holds for every $\phi\in C_c^{1}(\Omega)$ without a sign restriction.
\end{Definition}

\begin{Remark}\label{welldef}
Note that by Lemma \ref{defineq} and Lemma \ref{locnon1}, it follows that Definition \ref{wksoldef} is well stated.
\end{Remark}

\begin{Remark}\label{tstrmk}
Let  $u\in H_0^{1}(\Omega)$ be a weak solution of the problem \eqref{maineqn} when $g$ is either of the form $(g_1)$ or $(g_2)$. Then following the lines of the proof of \cite[Lemma 5.1]{GU}, it follows that the equality in \eqref{wksoleqn} holds, for every $\phi\in H_0^{1}(\Omega)$.
\end{Remark}

\subsection{Statement of the main results:} Our main results in this article reads as follows:

\begin{Theorem}\label{mthm}
Let $0<\gamma<1$ and $g$ be of
the form $(g_1)$. Then for every $\lambda>0$, there exists a weak solution $u\in H_0^{1}(\Omega)\cap L^{\infty}(\Omega)$ of
the problem \eqref{maineqn}.
\end{Theorem}

\begin{Theorem}\label{pu2thm}
Let $0<\gamma<1$ and $g$ be of the form $(g_2)$. Then there exists $\Lambda>0$ such that for every $\lambda\in(0,\Lambda)$ the problem \eqref{maineqn} admits at least two different
weak solutions in $H_0^{1}(\Omega)$.
\end{Theorem}

To prove our main results stated above, the following result concerning the mixed local and nonlocal eigenvalue problem \eqref{evp} will be useful for us.
\begin{equation}\label{evp}
-\Delta u+(-\Delta)^s u=\lambda |u|^{p-2}u\text{ in }\Omega,\quad u=0\text{ in }\mathbb{R}^n\setminus \Omega.
\end{equation}

\begin{Lemma}\label{evpthm}
(i) There exists the least eigenvalue $\lambda_1>0$ and at least one corresponding eigenfunction $e_1\in H_0^{1}(\Omega)\cap L^\infty(\Omega)\setminus\{0\}$ which is nonnegative in $\Om$. (ii) Moreover, for every $\omega\Subset\Om$, there exists a positive constant $c(\omega)$ such that $e_1\geq c(\omega)>0$ in $\omega$. 
\end{Lemma}
\begin{proof}
Part $(i)$ follows from \cite[Prop 2.6 and Theorem 2.8]{BDVV2}. 
Part $(ii)$ follows from \cite[Theorem 8.4]{GK}.
\end{proof}

Singular problems has drawn a great attention over the last three decade. Equations of the form
\begin{equation}\label{ln}
-\alpha\Delta u+\beta(-\Delta)^s u=\lambda f(u)u^{-\gamma}+\mu u^r,
\end{equation}
where $\alpha,\beta,\lambda,\mu,r\geq 0,\,\gamma>0$ are parameters and $f$ is some given function, are studied widely in both the local ($\beta=0$) and nonlocal ($\alpha=0$) cases separately. Here the singularity is captured by the parameter $\gamma>0$. Indeed, the quasilinear analouge of the equation \eqref{ln} is also investigated in the separate local and nonlocal cases and there is a colossal amount of work done for such problems.

More precisely, in the local case ($\beta=0$), Crandall-Rabinowitz-Tartar \cite{CRT} proved existence of classical solution of \eqref{ln} for $\lambda=1,\mu=0$ and $f(u)=1$ for any $\gamma>0$. Further, for a certain range of $\gamma$, Lazer-McKenna \cite{LazerMc} studied the notion of weak solutions. Boccardo-Orsina \cite{Boc} removed this restriction on $\gamma$ and proved existence of weak solutions for any $\gamma>0$. This study has further been investigated in the quasilinear setting by Canino-Sciunzi-Trombetta \cite{Canino}, see also De Cave \cite{DeCave} and the references therein. When $f(u)\geq 0$ and $\mu=0$, for $0<\gamma<1$ and a certain range of $\lambda$, equation \eqref{ln} is investigated by Ko-Lee-Shivaji in \cite{Lee}. In the purturbed case, we refer to Haitao \cite{Haitao}, Hirano-Saccon-Shioji \cite{Shi}, Arcoya-Boccardo-M\'{e}rida \cite{Ar-Boc, Ar-Mer}, Bal-Garain \cite{BG}, Giacomoni-Schindler-Tak\'{a}\v{c} in \cite{GST},  and the references therein.

In the nonlocal case ($\alpha=0$), equation \eqref{ln} is studied by Fang \cite{Fang} for $\mu=0$ and further been extended in the quasilinear setting by Canino-Montoro-Sciunzi-Squassina \cite{Caninononloc}. The perturbed  singular case ($\mu>0$) is investigated by Barrios-De Bonis-Medina-Peral \cite{BDMP}, Adimurthi-Giacomoni-Santra \cite{Adi}, Giacomoni-Mukherjee-Sreenadh \cite{GMS, GMS2} and generalized by Mukherjee-Sreenadh \cite{MS} in the quasilinear case and the references therein.

To the best of our knowledge, singular problems in the mixed local and nonlocal setting is very less known. Our main purpose in this article is to contribute in this topic. We believe it would be an interesting topic of further investigation. We would like to mention that mixed problems are also less known even in the nonsingular case. Using probability theory, Foondun \cite{Fo}, Chen-Kim-Song-Vondra\v{c}ek \cite{CKSV} studied regularity results for the equation
\begin{equation}\label{mln}
-\Delta u+(-\Delta)^s u=0.
\end{equation}
Recently based on purely analytic approach, Biagi-Dipierro-Salort-Valdinoci-Vecchi \cite{Biagi2, BDVV, SV} studied existence and regularity results for the mixed equation \eqref{mln}. Equation \eqref{mln} is also studied using analytic approach in the quasilinear case by Garain-Kinnunen \cite{GK}. Several recent regularity results and other qualitative properties for such problems using analytic approach can be found in see \cite{BDVV2, Biagi1, BVDV, DeMin, GL} and the references therein.

In the mixed singular case, that is for positive $\alpha$ and $\beta$, assuming $\mu=0$ and $f$ depending on $x$ only, the singular equation \eqref{ln} and its quasilinear version is studied recently. In this concern, for the quasilinear case, we refer to Garain-Ukhlov \cite{GU} for existence, uniqueness, regularity and symmetry properties with any $\gamma>0$. Further, associated extremal functions are also studied in \cite{GU}. Moreover, Arora-Radulescu \cite{Aro-Rad} studied several existence and regularity properties (which shows power and exponential type Sobolev regularity depending
upon the summability of the datum $f$ and the singular exponent $\gamma>0$) for the semilinear equation \eqref{ln}, where the case $\gamma=0$ is also considered.

{In this article, we establish existence and multiplicity results for the mixed problem \eqref{maineqn} where the singularity $g$ is either of type $(g_1)$ or $(g_2)$. We would like to emphasis that our main results for the mixed case (Theorem \ref{mthm} and Theorem \ref{pu2thm}) are similar to the associated Laplace equation, see \cite{Ar-Boc, Gtmna}. Although it is worth to mention that the presence of the nonlocal operator in the mixed equation cannot be neglected and such nonlocal affect is one of the main obstacle, see \cite{BDVV}. To overcome this difficulty, we simultaneously employ the theory developed for the Laplacian and fractional Laplacian to study the mixed equation \eqref{maineqn}. Further, we will make use of some recent results for the mixed operator.

More precisely, the variational technique introduced for the local case in \cite{Lee} will be adopted to the mixed case for proving Theorem \ref{mthm}. To this end, we also borrow ideas from \cite{Haitao} to prove the sub-supersolution result (Lemma \ref{Subsuplemma}), where to deal with the nonlocal behavior of the equation, we used the technique from \cite{GMS2}. Finally, the eigenvalue problem \eqref{evp} and the purely singular problem \eqref{neqn} related to the mixed operator are used to construct subsolution and supersolutions, thanks to Lemma \ref{evpthm} and Lemma \ref{nlem}.

To prove Theorem \ref{pu2thm}, we utilise the variational approach introduced for the local case in Arcoya-Boccardo \cite{Ar-Boc} in combination with the technique from \cite{GMnonloc} to deal with the nonlocality. To this end, we obtain existence of multiple solutions of the associated approximate problem \eqref{apx}. This fact combined with an apriori estimate (Lemma \ref{apriori}) gives us the required result.}

\subsection{Notation and organization of the article}
Throughout the rest of the article, by $c$ or $C$, we mean a positive constant which may vary from line to line or even in the same line. The dependency of the constants $c$ or $C$ on the parameters $r_1,r_2,\ldots,r_k$ is denoted by $c(r_1,r_2,\ldots,r_k)$ or $C(r_1,r_2,\ldots,r_k)$. For $a\in\mathbb{R}$, we denote by $a^+=\max\{a,0\}$ and $a^-=\max\{-a,0\}$. We use the notation $2^*=\frac{2n}{n-2}$ if $n>2$ and $2^*=\infty$ if $n=2$.

In Section 2, we obtain some preliminary results and prove Theorem \ref{mthm}. Finally, in Section 3, we establish some useful results and prove Theorem \ref{pu2thm}.

\section{Preliminaries for the proof of Theorem \ref{mthm}}
Throughout this section, we assume $g$ is of the form $(g_1)$. First we obtain some useful results. Consider the energy functional $J_\la: H_0^{1}(\Omega) \to \mb R \cup \{\pm\infty\}$ defined by
\[J_\la(u)=\int_{\Om}G(x,\nabla u)+\iint_{\mathbb{R}^{2n}}F(x,y,u)\,dx dy-\la \int_{\Om}H(u)~dx  \]
where
$$
G(x,\nabla u)=\frac{1}{2}|\nabla u|^2\\
$$
$$
F(x,y,u)=\frac{|u(x)-u(y)|^2}{|x-y|^{n+2s}}
$$
and
\begin{equation*}
H(t)=\left\{
\begin{aligned}
&\int_0^t h(\tau)\tau^{-\gamma}~d\tau,\; \text{if}\; t>0,\\
&0, \; \text{if}\; t\leq 0.
\end{aligned}\right.
\end{equation*}
Following Haitao \cite{Haitao}, we establish the following result in the mixed local and nonlocal setting.
\begin{Lemma}\label{Subsuplemma}
Suppose that $\underline{u}, \overline{u} \in H_0^{1}(\Omega) \cap  L^\infty(\Om)$ are weak subsolution and supersolution of \eqref{maineqn} respectively such that $0<\underline{u} \leq \overline{u}$ in $\Omega$ and $\uline{u}\geq c(\omega)>0$ for every $\omega \Subset  \Om$, for some constant $c(\omega)$. Then there exists a weak solution $u\in H_0^{1}(\Omega) \cap L^\infty(\Om)$ of \eqref{maineqn} satisfying $\uline{u}\leq u\leq \oline{u}$ in $\Om$.
\end{Lemma}
\begin{proof}
Let us consider the set
$$
S=\{v\in H_0^{1}(\Omega):\uline{u}\leq v\leq \oline{u}\text{ in }\Om\}.
$$
Since $\uline{u}\leq \oline{u}$ in $\Om$, we have $S\neq \emptyset$. We observe that $S$ is closed and convex. We establish the result in the following two Steps. \\
\textbf{Step $1$:} We claim that $J_\la$ admits a minimizer $u$ over $S$.\\
To this end, we prove that $J_\la$ is weakly sequentially lower semicontinuous over $S$. Indeed, let $\{v_k\}_{k\in\mathbb{N}} \subset S$ be such that $v_k \rightharpoonup v$ weakly in $H_0^{1}(\Omega)$. Then by the hypothesis on $h$, we have
\[H(v_k) \leq \int_0^{\oline{u}} h(\tau)\tau^{-\gamma}~d\tau \leq \frac{h(\|\oline{u}\|_\infty)}{(1-\gamma)}\|\oline{u}\|_\infty^{1-\gamma}.\]
Therefore by the Lebesgue's Dominated Convergence theorem and weak lower semicontinuity of norm, the claim follows. Hence, there exists a minimizer $u \in S$ of $J_\la$ that is $J_\la(u)= \inf\limits_{v\in S}J_\la(v)$.\\
\textbf{Step $2$:} Here, we prove that $u$ is a weak solution of \eqref{maineqn}.\\
Let $\phi \in C_c^{1}(\Om)$ and $\epsilon>0$. We define
\begin{equation*}
\eta_{\epsilon}=\left\{
\begin{aligned}
&\oline{u} \;\;\;\text{if} \; u+\epsilon \phi \geq \oline{u},\\
&u+\epsilon \phi  \;\;\;\text{if}\; \uline{u}\leq u+\epsilon \phi \leq \oline{u},\\
&\uline{u} \;\;\;\text{if} \; u+\epsilon \phi \leq \uline{u}.
\end{aligned}\right.
\end{equation*}
Observe that $\eta_{\epsilon}=u+\epsilon\phi-\phi^{\epsilon}+\phi_{\epsilon}\in S$, where $\phi^{\epsilon}= (u+\epsilon \phi -\oline{u})^+$ and $\phi_{\epsilon}= (u+\epsilon \phi -\uline{u})^-$. By Step $1$ above, since $u$ is a minimizer of $J_\la$, we have
\begin{equation}\label{eq1}
\begin{split}
0 & \leq \lim_{t \to 0} \frac{J_\la(u+t(\eta_{\epsilon}-u))- J_\la(u)}{t}= I_1+I_2-\la J \;\text{(say)},
\end{split}
\end{equation}
with
 \[I_1= \int_\Om\nabla u\nabla (\eta_\epsilon -u)~dx,\]
 \[I_{2}=\int_{Q}(\eta_{\epsilon}-u)(-\Delta)^s u\,dx,\]
 \[
  J=\int_{\Om}{(\eta_\epsilon -u){u^{-\gamma}}h(u)}\,dx,
\]
where we have used the notation $Q=\mathbb{R}^{2n}\setminus(\mathcal{C}\Om\times\mathcal{C}\Om)$, where $\mathcal{C}\Om:=\mathbb{R}^n\setminus\Om$.
Therefore, we have
\begin{equation}\label{eq2}
\begin{split}
&0\leq  \int_\Om \nabla u\nabla(\eta_{\epsilon}-u)~dx+\int_{Q}(\eta_{\epsilon}-u)(-\Delta)^su\,dx - \la \int_\Om {(\eta_\epsilon -u){u^{-\gamma}}h(u)}~dx\\
& \implies \frac{1}{\epsilon}(Q^\epsilon -Q_\epsilon)\leq \int_\Om\nabla u\nabla \phi~dx+\int_{\mathbb{R}^n}\phi(-\Delta)^s u\,dx - \la \int_\Om {{u^{-\gamma}}h(u)}\phi~dx
\end{split}
\end{equation}
where
\[Q^\epsilon= \int_\Om\nabla u\nabla\phi^{\epsilon}~dx+\int_{\mathbb{R}^n}\phi^{\epsilon}(-\Delta)^s u\,dx-\la \int_\Om {{u^{-\gamma}}h(u)}\phi^{\epsilon}~dx \]
\[\text{and}\quad\;Q_\epsilon= \int_\Om\nabla u\nabla\phi_{\epsilon}~dx+\int_{\mathbb{R}^n}\phi_{\epsilon}(-\Delta)^s u\,dx-\la \int_\Om {{u^{-\gamma}}h(u)}\phi_{\epsilon}~dx .\]

\textbf{Estimate of $Q^{\epsilon}$:} We observe that
\begin{equation}\label{l1}
\begin{split}
\frac{1}{\epsilon}\int_\Om\nabla u\nabla\phi^{\epsilon}~dx&=\frac{1}{\epsilon}\int_{\Omega}\nabla(u-\overline{u})\nabla\phi^{\epsilon}\,dx\geq \int_{\Omega^{\epsilon}}\nabla(u-\overline{u})\nabla\phi\,dx+\frac{1}{\epsilon}\int_{\Omega}\nabla\overline{u}\nabla\phi^{\epsilon}\,dx\Big)\\
&\geq o(1)+\frac{1}{\epsilon}\int_{\Omega}\nabla\overline{u}\nabla\phi^{\epsilon}\,dx.
\end{split}
\end{equation}

Further, we notice that
\begin{equation}\label{nl1}
\begin{split}
\frac{1}{\epsilon}\int_{\mathbb{R}^n}\phi^{\epsilon}(-\Delta)^s u\,dx&=\frac{1}{\epsilon}\Big(\int_{\mathbb{R}^n}\phi^{\epsilon}(-\Delta)^s (u-\overline{u})\,dx+\int_{\mathbb{R}^n}\phi^{\epsilon}(-\Delta)^s \overline{u}\,dx\Big)\\
&\geq o(1)+\frac{1}{\epsilon}\int_{\mathbb{R}^n}\phi^{\epsilon}(-\Delta)^s \overline{u}\,dx,
\end{split}
\end{equation}
where to estimate the last inequality, we used the the lines of the proof from \cite[Page 9]{GMS2}. Combining \eqref{l1} and \eqref{nl1}, we have
\begin{equation}\label{eps1}
\begin{split}
\frac{1}{\epsilon}Q^{\epsilon}&\geq o(1)+\frac{1}{\epsilon}\Big(\int_{\Omega}\nabla\overline{u}\nabla\phi^{\epsilon}\,dx+\int_{\mathbb{R}^n}\phi^{\epsilon}(-\Delta)^s \overline{u}\,dx-\la \int_\Om {{u^{-\gamma}}h(u)}\phi^{\epsilon}~dx\Big)\\
&=o(1)+\frac{1}{\epsilon}\Big(\int_{\Omega}\nabla\overline{u}\nabla\phi^{\epsilon}\,dx+\int_{\mathbb{R}^n}\phi^{\epsilon}(-\Delta)^s \overline{u}\,dx-\la \int_\Om {{\overline{u}^{-\gamma}}h(\overline{u})}\phi^{\epsilon}~dx\Big)\\
&\quad+\frac{\lambda}{\epsilon}\Big(\int_\Om {{\overline{u}^{-\gamma}}h(\overline{u})}\phi^{\epsilon}~dx-\int_\Om {{{u}^{-\gamma}}h({u})}\phi^{\epsilon}~dx\Big)\\
&\geq o(1)+\frac{\la}{\epsilon}\int_{\Om^{\epsilon}}h(u)(\oline{u}^{-\gamma}-u^{-\gamma})(u-\oline{u})\,dx\\
&\quad \quad+\la\int_{\Om^{\epsilon}}h(u)(\oline{u}^{-\gamma}-u^{-\gamma})\phi\,dx\\
&\geq o(1),
\end{split}
\end{equation}
using that $\oline{u}$ is a weak supersolution of \eqref{maineqn}, $u\leq\oline{u}$ and $\displaystyle\int_{\Om^{\epsilon}}h(u)(\oline{u}^{-\gamma}-u^{-\gamma})\phi\,dx\leq{2{c(\omega)^{-\gamma}}h(||\oline{u}||_{\infty})}||\phi||_{\infty}<+\infty$, where $\Om^{\epsilon}= \text{supp}\;\phi^{\epsilon}$ and $\omega=\mathrm{supp}\,\phi$.

Taking into account that $\uline{u}$ is a weak subsolution of \eqref{maineqn}, $u \geq \uline{u}$ and $\displaystyle \int_{\Om_{\epsilon}}h(u)(\uline{u}^{-\gamma}-u^{-\gamma})\phi\,dx\leq {2c(\omega)^{-\gamma} h(\|\oline{u}\|_\infty)}\|\phi\|_\infty<+\infty$, where $\Om_\epsilon=\mathrm{supp}\,\phi_\epsilon$ and $\omega=\mathrm{supp}\,\phi$, in a similar way, we obtain
\begin{equation}\label{eps2}
\frac{1}{\epsilon}Q_\epsilon\leq o(1).
\end{equation}

Using the estimates \eqref{eps1} and \eqref{eps2} in \eqref{eq2}, we conclude that
\[0 \leq\int_\Om\nabla u\nabla \phi~dx+\int_{\mathbb{R}^n}\phi(-\Delta)^s u\,dx - \la \int_\Om {{u^{-\gamma}}h(u)}\phi~dx.\]
Since $\phi\in C_c^{1}(\Om)$ is arbitrary, our claim follows. This completes the proof.
\end{proof}

\begin{Lemma}\label{nlem}
Let $0<\gamma<1$ and $v_0\in H_0^{1}(\Om)$ be a weak solution of the problem 
\begin{equation}\label{neqn}
-\Delta u+(-\Delta)^s u=u^{-\gamma}\text{ in }\Om,\quad u>0\text{ in }\Om,\quad u=0\text{ in }\mathbb{R}^n\setminus\Om.
\end{equation}
Then $v_0\in L^\infty(\Om)$.
\end{Lemma}
\begin{proof}
Let $k>1$, then by Remark \ref{tstrmk} we choose $\phi_k = (v_0-k)^+ \in H_0^{1}(\Om)$ as a test function in (\ref{neqn}) and apply H$\ddot{\text{o}}$lder's along with Young's inequality with $\epsilon\in(0,1)$ to get
\begin{align*}
\int_{\Omega}|\nabla\phi_k|^2\,dx\leq\,C(\epsilon)|A(k)|^\frac{2}{q^{'}} + \epsilon\int_{\Omega}|\nabla\phi_k|^2\,dx,
\end{align*}
where $A(k)=\big\{x\in\Omega:v_0\geq k \text{ in }\Omega\big\}$. In the above estimate, we have also used that $H_0^{1}(\Om)\hookrightarrow L^q(\Om)$ for some $q>2$ from Lemma \ref{emb}. Therefore, fixing $\epsilon\in(0,1)$, we obtain
$$
\int_{\Omega}|\nabla\phi_k|^2\,dx \leq C|A(k)|^\frac{2}{q^{'}},
$$
where $C$ is some positive constant. Let $1 < k < h$, then since $A(h)\subset A(k)$, we have
\begin{align*}
(h-k)^p|A(h)|^\frac{2}{q}
\leq\Big(\int_{A(h)}(v_0-k)^{q}\,dx\Big)^\frac{2}{q}&\leq\Big(\int_{A(k)}(v_0-k)^{q}\,dx\Big)^\frac{2}{q}\\
&\leq C\int_{\Omega}|\nabla\phi_k|^2\,dx\leq C\,|A(k)|^\frac{2}{q^{'}}.
\end{align*}
Therefore
$$
|A(h)| \leq\frac{C}{(h-k)^q}|A(k)|^{q-1}.
$$
Since $q>2$, by \cite[Lemma B.1]{Stam}, we have  
$
||v_0||_{L^\infty(\Omega)} \leq c,
$
where $c$ is a positive constant. Hence the result follows.
\end{proof}

\subsection{Proof of Theorem \ref{mthm}}
We construct a pair of weak subsolution and supersolution of \eqref{maineqn} according to Lemma \ref{Subsuplemma}. By Lemma \ref{evpthm}, there exists $e_1\in H_0^{1}(\Omega)\cap L^\infty(\Omega)$ such that
\begin{equation}\label{aevp}
-\Delta e_1+(-\Delta)^s e_1=\la_1 e_1 \; \text{in}\; \Om, \;\;e_1>0\text{ in }\Om,\;\; e_1=0\;\text{in}\; \mathbb{R}^n\setminus\Om
\end{equation}
and for every $\omega\Subset\Omega$, there exists a positive constant $c(\omega)$ with $e_1\geq c(\omega)$ in $\omega$. {By $(h_2)$, we know that $\lim\limits_{t\to 0} t^{-\gamma}{h(t)}=\infty$, so we can choose $a_\la>0$ sufficiently small such that
\begin{equation}\label{ala}
\la_1 (a_\la e_1) \leq \la {(a_\la e_1)^{-\gamma}}{h(a_\la e_1)}.
\end{equation}
Let $\uline{u}=a_{\la}e_1$, then $\uline{u}\in H_0^{1}(\Omega)\cap L^\infty(\Omega)$ and by \eqref{aevp} and \eqref{ala}, we get
\begin{equation}\label{sub}
-\Delta {\uline{u}}+(-\Delta)^s {\uline{u}} \leq \la {(a_\la e_1)^{-\gamma}}{h(a_\la e_1)}=\la {\uline{u}^{-\gamma}}{h(\uline{u})}\;\text{in}\; \Om.
\end{equation}
By \cite[Theorem 2.13]{GU} and Lemma \ref{nlem}, there exists $v_0\in H_0^{1}(\Omega)\cap L^\infty(\Om)$ such that for every $\omega\Subset\Omega$ there exists a positive constant $c(\omega)$ satisfying $v_0\geq c(\omega)>0$ in $\omega$ and
\begin{equation}\label{fs}
-\Delta v_0+(-\Delta)^s v_0 = v_0^{-\gamma}\text{ in } \Om,\; v_0>0 \; \text{in}\; \Om,\; v_0=0\;\text{in}\; \mathbb{R}^n\setminus\Om.
\end{equation}
By the hypothesis $(h_2)$, since $\lim\limits_{t\to \infty}t^{-(\gamma+1)}{h(t)}=0$, we choose $b_\la>0$ sufficiently large such that
\begin{equation}\label{bla}
{(b_\la \|v_0\|_\infty)^{-(\gamma+1)}}{h({b_\la \|v_0\|_\infty)}}\leq \frac{1}{\la \|v_0\|^{\gamma+1}_\infty}.
\end{equation}
We define $\oline{u}:= b_\la v_0$. Then $\oline{u}\in H_0^{1}(\Omega)\cap L^\infty(\Omega)$ and using \eqref{fs} and \eqref{bla}, we have
\begin{equation}\label{sup}
-\Delta {\oline{u}} +(-\Delta)^s {\oline{u}} = v_0^{-\gamma}{b_\la}  \geq \la{(b_\la v_0)^{-\gamma}}{h({b_\la \|v_0\|_\infty)}} \geq \la{\oline{u}^{-\gamma}} {h(\oline{u})}\; \text{in}\; \Om,
\end{equation}
where we have also used the nondecreasing property of $h$ from $(h_1)$.
Thus, from \eqref{sub} and \eqref{sup}, it follows that $\uline{u}$ and $\oline{u}$ are weak subsolution and supersolution of \eqref{maineqn} respectively and the constants $a_\la, b_\la$ can be chosen in such a way that $\uline{u}\leq \oline{u}$. Therefore, by Lemma \ref{Subsuplemma}, the result follows.

\section{Preliminaries for the proof of Theorem \ref{pu2thm}}
In this section, we consider the equation \eqref{maineqn} when $g$ is of the form $(g_2)$, which reads as
\begin{equation}\label{psp}
\begin{split}
-\Delta u+(-\Delta)^s u&=\lambda u^{-\gamma}+u^q\text{ in }\Omega,\quad u>0\text{ in }\Omega,\quad u=0\text{ in }\mathbb{R}^n\setminus\Omega,
\end{split}
\end{equation}
where $\lambda>0$, $0<\gamma<1$ and  $q\in(1,2^*-1)$ where $2^*=\frac{2n}{n-2}$ if $n>2$ and $2^*=\infty$ if $n=2$. To this end, we study the functional $I_\la: H_0^{1}(\Omega)\to\mb R\cup\{\pm \infty\}$ associated with the problem \eqref{psp} given by
\begin{equation}\label{fnl1}
I_\la(u) :=\frac{1}{2}\int_{\Omega}|\nabla u|^2\,dx +\frac{1}{2}\iint_{\mathbb{R}^{2n}}\frac{|u(x)-u(y)|^2}{|x-y|^{n+2s}}\,dx dy-\la \int_\Om \frac{(u^+)^{1-\gamma}}{1-\gamma}~dx -\frac{1}{q+1}\int_\Om (u^+)^{q+1}~dx.
\end{equation}
For $\epsilon>0$, we consider the approximated problem
\begin{equation}\label{apx}
\begin{aligned}
  -\Delta u+(-\Delta)^s u &=\la(u^{+} +\epsilon)^{-\gamma}+ (u^+)^q\;\text{in}\; \Om,\quad u=0 \; \text{ in }\; \mathbb{R}^n\setminus\Om.
\end{aligned}
\end{equation}
We remark that the energy functional associated with the problem \eqref{apx} is given by
\begin{equation}\label{fnl2}
\begin{split}
I_{\la,\epsilon}(u) &= \frac{1}{2}\int_{\Omega}|\nabla u|^2\,dx +\frac{1}{2}\iint_{\mathbb{R}^{2n}}\frac{|u(x)-u(y)|^2}{|x-y|^{n+2s}}\,dx dy -\la\int_\Om \frac{[(u^+ +\epsilon)^{1-\gamma}-\epsilon^{1-\gamma}]}{1-\gamma}~dx\\
&\quad\quad\quad\quad\quad\quad\quad\quad\quad\quad\quad\quad\quad\quad\quad\quad\quad\quad\quad\quad-\frac{1}{q+1}\int_\Om (u^+)^{q+1}~dx.
\end{split}
\end{equation}
We observe that $I_{\la,\epsilon}\in C^1\big(H_0^{1}(\Omega),\mb R\big)$, $I_{\la,\epsilon}(0)=0$ and $I_{\la,\epsilon}(v)\leq I_{0,\epsilon}(v)$, for all $ v \in H_0^{1}(\Omega)$. Let us define
\begin{equation}\label{l}
l=
\begin{cases}
2^{*}=\frac{2n}{n-2},\text{ if } n>2,\\
r,\text{ if }n=2,
\end{cases}
\end{equation}
where $r>1$ is such that $1<q<r-1$ if $n=2$. Next we prove that $I_{\la,\epsilon}$ satisfies the Mountain Pass Geometry.

\begin{Lemma}\label{MP-geo}
There exists $R>0,\,\rho>0$ and $\Lambda>0$ depending on $R$ such that
$$
\inf\limits_{\|v\|\leq R}I_{\la,\epsilon}(v)<0\;\text{and}\;
\inf\limits_{\|v\|=R}I_{\la,\epsilon}(v)\geq \rho,\text{ for }\la\in(0,\Lambda).
$$
Moreover, there exists $T>R$ such that
$
I_{\la,\epsilon}(Te_{1})<-1$ for $\la\in (0,\La)$, where $e_1$ is given by Lemma \ref{evpthm}.
\end{Lemma}

\begin{proof}
Recalling the definition of $l$ from \eqref{l}, we define $\theta=|\Om|^{\frac{1}{\left(\frac{l}{q+1}\right)'}}$. By H\"{o}lder's inequality and Lemma \ref{emb}, for every $v\in H_0^{1}(\Omega)$, we have
\begin{equation}\label{MP1}
\int_\Om (v^+)^{q+1}~dx \leq \left( \int_\Om |v|^{l}\right)^{\frac{q+1}{l}} |\Om|^{\frac{1}{(\frac{l}{q+1})'}}\leq C\theta\|v\|^{q+1},
\end{equation}
for some positive constant $C$ independent of $v$. Since
$$
\lim_{t\to 0}\frac{I_{\la,\epsilon}(te_1)}{t}=-\la\int_{\Omega}\epsilon^{-\gamma}e_{1}\,dx<0,
$$
we choose $k\in(0,1)$ sufficiently small and set $\|v\|=R :=k(\frac{q+1}{pC\theta})^\frac{1}{q-1}$ such that 
$$
\inf\limits_{\|v\|\leq R}I_{\la,\epsilon}(v)<0.
$$
Moreover, using the fact $R<(\frac{q+1}{pC\theta})^\frac{1}{q-1}$ and the estimate \eqref{MP1}, we have
\begin{equation}\label{up}
I_{0,\epsilon}(v)\geq \frac{R^2}{2}-\frac{C\theta R^{q+1}}{q+1}
:=
2\rho\,(\text{say})>0.
\end{equation}
We define $$\Lambda:=\frac{\rho}{\sup\limits_{\|v\|=R} \left(\displaystyle\frac{1}{1-\gamma}\int_\Om |v|^{1-\gamma}~dx \right)},$$
which is positive. Note that, since $\rho,R$ depends on $k,q,|\Omega|$ and $C$, so does $\Lambda$. We observe that
\begin{equation}\label{knownfact}
(v^{+}+\epsilon)^{1-\gamma}-\epsilon^{1-\gamma}\leq (v^+)^{1-\gamma}.
\end{equation}
Therefore, we have
\begin{align*}
I_{\lambda,\epsilon}(v)&\geq \frac{1}{p}\int_{\Om}|\nabla v|^2\,dx+\iint_{\mathbb{R}^{2n}}\frac{|v(x)-v(y)|^2}{|x-y|^{n+2s}}\, dx dy-\frac{1}{q+1}\int_{\Om}(v^{+})^{q+1}\,dx-\frac{\la}{1-\gamma}\int_{\Om}(v^{+})^{1-\gamma}\,dx\\
&= I_{0,\epsilon}(v)-\frac{\la}{1-\gamma}\int_{\Om}(v^{+})^{1-\gamma}\,dx.
\end{align*}
Hence, using \eqref{up}, for $\la\in(0,\Lambda)$, we get
\begin{align*}
\inf\limits_{\|v\|=R} I_{\la,\epsilon}(v)&\geq\inf\limits_{\|v\|=R}I_{0,\epsilon}(v)-\la \sup\limits_{\|v\|=R} \left(\frac{1}{1-\gamma}\int_\Om |v|^{1-\gamma}~dx \right)\\
&\geq 2\rho -\la \sup\limits_{\|v\|=R} \left(\frac{1}{1-\gamma}\int_\Om |v|^{1-\gamma}~dx \right)\geq \rho.
\end{align*}
Finally, we observe that $I_{0,\epsilon}(te_1) \to -\infty$, as $t\to +\infty$. This gives the existence of $T>R$ such that $I_{0,\epsilon}(Te_1)<-1$. Therefore, 
\[I_{\la,\epsilon}(Te_1)\leq I_{0,\epsilon}(Te_1)<-1,\]
which completes the proof.
\end{proof}

Next, we prove that $I_{\la,\epsilon}$ satisfies the Palais Smale $(PS)_c$ condition.

\begin{Lemma}\label{PS-cond}
$I_{\la,\epsilon}$ satisfies the $(PS)_c$ condition, for any $c \in \mb R$, that is if $\{u_k\}_{k\in\mathbb{N}}\subset H_0^{1}(\Omega)$ is a sequence such that
\begin{equation}\label{PS1}
I_{\la,\epsilon}(u_k)\to c \; \text{and}\; I_{\la,\epsilon}^\prime(u_k) \to 0
\end{equation}
as $k \to \infty$, then $\{u_k\}_{k\in\mathbb{N}}$ contains a strongly convergent subsequence in $H_0^{1}(\Omega)$.
\end{Lemma}
\begin{proof}
We prove the result in two steps below.\\
\textbf{Step $1$.} First, we claim that if $\{u_k\}_{k\in\mathbb{N}} \subset H_0^{1}(\Omega)$ satisfies \eqref{PS1} then $\{u_k\}_{k\in\mathbb{N}}$ is uniformly bounded in $H_0^{1}(\Omega)$. To this end, by \eqref{knownfact}, for some positive constant $C$ (independent of $k$), we have
\begin{equation}\label{PS2}
\begin{split}
&I_{\la,\epsilon}(u_k)- \frac{1}{q+1}I_{\la,\epsilon}^\prime(u_k)u_k = \left( \frac{1}{2}-\frac{1}{q+1}\right)
\int_{\Omega}|\nabla u_k|^2\,dx+\left( \frac{1}{2}-\frac{1}{q+1}\right)
\iint_{\mathbb{R}^{2n}}\frac{|u(x)-u(y)|^2}{|x-y|^{n+2s}}\,dx dy\\
&\quad-{\la}\int_\Om \frac{(u_k^+ +\epsilon)^{1-\gamma}-\epsilon^{1-\gamma}}{1-\gamma}~dx +\frac{\la}{q+1}\int_\Om (u_k^+ +\epsilon)^{-\gamma}u_k~dx\\
& \geq \left( \frac{1}{2}-\frac{1}{q+1}\right)\|u_k\|^2-C\|u_k\|^{1-\gamma},
\end{split}
\end{equation}
for some positive constant $C$ (independent of $k$), where we have also used Lemma \ref{emb} and H\"older's inequality. 
Noting $q>1$ and using \eqref{PS2}, we obtain
\begin{equation}\label{PS2-new1}
I_{\la,\epsilon}(u_k)- \frac{1}{q+1}I_{\la,\epsilon}^\prime(u_k)u_k \geq C_1\|u_k\|^2 -C\|u_k\|^{1-\gamma},
\end{equation}
for some positive constants $C,C_1$ (independent of $k$). Using \eqref{PS1}, for $k$ large enough, we have
\begin{equation}\label{PS3}
\left| I_{\la,\epsilon}(u_k)- \frac{1}{q+1}I_{\la,\epsilon}^\prime(u_k)u_k\right| \leq C+o(\|u_k\|),
\end{equation}
for some positive constant $C$ (independent of $k$). Combining \eqref{PS2-new1} and \eqref{PS3}, our claim follows. \\
\textbf{Step $2$.} We claim that up to a subsequence, $u_k \to u_0$ strongly in $H_0^{1}(\Omega)$ as $k \to \infty$.\\
By Step $1$, since $\{u_k\}_{k\in\mathbb{N}}$ is uniformly bounded in $H_0^{1}(\Omega)$, due to the reflexivity of $H_0^{1}(\Omega)$, there exists $u_0\in H_0^{1}(\Omega)$ such that up to a subsequence, $u_k \rightharpoonup u_0$ weakly in $H_0^{1}(\Omega)$ as $k \to \infty$. Again, by \eqref{PS1}, we have
\begin{equation*}
\begin{split}
\lim_{k\to \infty}\Bigg(\mc \int_{\Om}\nabla u_k\nabla u_0\,dx+\iint_{\mathbb{R}^{2n}}\frac{(u_k(x)-u_k(y))(u_0(x)-u_0(y))}{|x-y|^{n+2s}}\,dx dy\\
-\la \int_\Om (u_k^+ +\epsilon)^{-\gamma}u_0~dx - \int_\Om (u_k^+)^{q} u_0~dx\Bigg)=0
\end{split}
\end{equation*}
and
\[\lim_{k\to \infty}\Bigg(\mc \int_{\Om}|\nabla u_k|^2\,dx+\iint_{\mathbb{R}^{2n}}\frac{|u_k(x)-u_k(y)|^2}{|x-y|^{n+2s}}\,dx dy- \la \int_\Om (u_k^+ +\epsilon)^{-\gamma}u_k~dx - \int_\Om (u_k^+)^q u_k~dx\Bigg)=0.\]
The preceding two inequalities give,
\begin{equation}\label{PS4}
\begin{split}
&\lim\limits_{k\to\infty}\left(\int_{\Omega}|\nabla(u_k-u_0)|^2\,dx+\iint_{\mathbb{R}^{2n}}\frac{|(u_k(x)-u_k(y))-(u_0(x)-u_0(y))|^2}{|x-y|^{n+2s}}\, dx dy\right)\\
&=\lim\limits_{k\to\infty} \left( \la \int_\Om (u_k^+ +\epsilon)^{-\gamma}u_k~dx + \int_\Om (u_k^+)^q u_k~dx - \la \int_\Om (u_k^+ +\epsilon)^{-\gamma}u_0~dx - \int_\Om (u_k^+)^q u_0~dx\right)\\
&\quad -\lim_{k\to \infty}\left(\int_\Om  \nabla u_0\nabla u_k~dx - \int_\Om |\nabla u_0|^2~dx\right)\\
&\quad -\lim_{k\to\infty}\left(\iint_{\mathbb{R}^{2n}}\frac{(u_0(x)-u_0(y)) (u_k(x)-u_k(y))}{|x-y|^{n+2s}}\,dx dy-\iint_{\mathbb{R}^{2n}}\frac{|u_0(x)-u_0(y)|^2}{|x-y|^{n+2s}}\,dx dy\right).
\end{split}
\end{equation}
Since $u_k \rightharpoonup u_0$ weakly in $H_0^{1}(\Omega)$ as $k \to \infty$, we observe that
\begin{equation}\label{PS5}
\lim_{k\to \infty}\left(\int_\Om  \nabla u_0\nabla u_k~dx - \int_\Om |\nabla u_0|^2~dx\right)=0.
\end{equation}
Further, since $u_k \rightharpoonup u_0$ weakly in $H_0^{1}(\Omega)$ as $k \to \infty$, it follows that
\begin{equation}\label{PS5new}
\lim_{k\to\infty}\left(\iint_{\mathbb{R}^{2n}}\frac{(u_0(x)-u_0(y)) (u_k(x)-u_k(y))}{|x-y|^{n+2s}}\,dx dy-\iint_{\mathbb{R}^{2n}}\frac{|u_0(x)-u_0(y)|^2}{|x-y|^{n+2s}}\,dx dy\right)=0.
\end{equation}
Indeed, the weak convergence of $u_k$ to $u_0$ implies that
$$
\frac{u_k(x)-u_k(y)}{|x-y|^{n+2s}}\rightharpoonup\frac{u_0(x)-u_0(y)}{|x-y|^{n+2s}}\quad\text{weakly\,in}\quad L^2(\mathbb{R}^{2n}),
$$
which combined with the fact that
$$
\frac{u_0(x)-u_0(y)}{|x-y|^\frac{n+2s}{2}}\in L^2(\mathbb{R}^{2n})
$$
proves \eqref{PS5new}.

On the other hand, since
\begin{align*}
\left|(u_k^++\epsilon)^{-\gamma}u_0\right| \leq\epsilon^{-\gamma}u_0\text{ and }
\int_\Om \left|\epsilon^{-\gamma}u_0\right|dx \leq \epsilon^{-\gamma}\int_\Om|u_0|~dx< +\infty,
\end{align*}
by the Lebesgue's Dominated convergence theorem, it follows that
\begin{equation}\label{PS6}
\lim_{k \to \infty} \int_\Om (u_k^+ +\epsilon)^{-\gamma}u_0~dx = \int_\Om (u_0^+ +\epsilon)^{-\gamma}u_0~dx.
\end{equation}
Since $u_k \to u_0$ pointwise almost everywhere in $\Om$ and for any measurable subset $E$ of $\Om$,
\begin{equation*}
\begin{split}
\int_E |(u_k^++\epsilon)^{-\gamma}u_k |~dx&\leq \int_E\epsilon^{-\gamma}|u_k|~dx\leq\|\epsilon^{-\gamma}\|_{L^\infty(\Om)}\|u_k\|_{L^{l}(\Om)}|E|^{\frac{l-1}{l}}\leq C(\epsilon)|E|^{\frac{l-1}{l}},
\end{split}
\end{equation*}
using Vitali's convergence theorem, we have
\begin{equation}\label{PS7}
\lim\limits_{k\to\infty} \la \int_\Om (u_k^+ +\epsilon)^{-\gamma}u_k~dx = \la \int_\Om (u_0^+ +\epsilon)^{-\gamma}u_0~dx.
\end{equation}
Since $q+1<l$, we have
\[\int_E |(u_k^+)^q u_0|~dx \leq \|u_0\|_{L^{l}(\Om)} \left(\int_E (u_k^+)^{ql^{'}}~dx\right)^{\frac{1}{l{'}}}\leq C_3 |E|^{\alpha}  \]
and
\[\int_E |(u_k^+)^q u_k|~dx \leq \|u_k\|_{L^{l}(\Om)} \left(\int_E (u_k^+)^{ql{'}}~dx\right)^{\frac{1}{l{'}}}\leq C_4 |E|^{\beta}  \]
for some positive constants $C_3,C_4,\alpha$ and $\beta$. Again using Vitali's convergence theorem, we get
\begin{equation}\label{PS8}
\lim_{k \to \infty} \int_\Om (u_k^+)^qu_0~dx  =\int_\Om (u_0^+)^qu_0~dx,
\end{equation}
and
\begin{equation}\label{PS9}
\lim_{k \to \infty} \int_\Om (u_k^+)^qu_k~dx  =\int_\Om (u_0^+)^qu_0~dx.
\end{equation}
Using \eqref{PS5}, \eqref{PS5new}, \eqref{PS6}, \eqref{PS7}, \eqref{PS8} and \eqref{PS9} in \eqref{PS4}, we obtain $u_k\to u_0$ strongly in $H_0^{1}(\Omega)$ as $k\to\infty$ which proves our claim.
\end{proof}

\begin{Remark}\label{multrmk}
Using Lemma \ref{MP-geo}, Lemma \ref{PS-cond} and the Mountain Pass Lemma, for every $\la\in(0,\Lambda)$, there exists $\zeta_\epsilon \in H_0^{1}(\Omega)$ such that $I_{\lambda,\epsilon}^\prime(\zeta_\epsilon)=0$ {and}
$$
I_{\lambda,\epsilon}(\zeta_{\epsilon})=\inf_{\gamma\in\Gamma}\max_{t \in [0,1]}I_{\la,\epsilon}(\gamma (t)) \geq \rho >0,
$$
where 
$$
\Gamma =\big\{\gamma \in C([0,1],H_0^{1}(\Omega)):\gamma(0)=0, \gamma(1)=Te_1\big\}.
$$
Moreover, as a consequence of Lemma \ref{MP-geo}, since for every $\la\in(0,\Lambda)$ we have $\inf\limits_{\|v\|\leq R} I_{\la,\epsilon}(v)<0$, by the weak lower semicontinuity of $I_{\la,\epsilon}$, there exists a nonzero $\nu_\epsilon\in H_0^{1}(\Omega)$ such that $\|\nu_\epsilon\| \leq R$ and
\begin{equation}\label{limit-pass}
\inf\limits_{\|v\|\leq R} I_{\la,\epsilon}(v) =I_{\la,\epsilon}(\nu_\epsilon)<0<\rho \leq I_{\la,\epsilon}(\zeta_\epsilon).
\end{equation}
Thus, $\zeta_\epsilon$ and $\nu_\epsilon$ are two different non trivial critical points of $I_{\la,\epsilon}$, provided $\la\in(0,\Lambda)$.
\end{Remark}

\begin{Lemma}\label{non-negative}
The critical points $\zeta_\epsilon$ and $\nu_\epsilon$ of $I_{\la,\epsilon}$ are nonnegative in $\Omega.$
\end{Lemma}

\begin{proof}
Let $u=\zeta_\epsilon$ or $\nu_\epsilon$. Therefore, since the integrand
$
\la(u^+ +\epsilon)^{-\gamma}+(u^+)^q
$
is nonnegative in $\Om$, testing \eqref{apx} with $v=\min\{u,0\}$ and proceeding exactly as in the proof of \cite[Pages 11-12, Lemma 3.1]{GU} (or \cite[Page 11, Lemma 3.1]{Aro-Rad}), we get $u\geq 0$ in $\Om$. This completes the proof.
\end{proof}

\begin{Lemma}\label{apriori}
There exists a constant $\Theta>0$ (independent of $\epsilon$) such that $\|v_\epsilon\| \leq \Theta$, where $v_\epsilon = \zeta_\epsilon$ or $\nu_\epsilon$.
\end{Lemma}
\begin{proof}
We notice that the result trivially holds if $v_\epsilon = \nu_\epsilon$. Thus, it is enough to deal with the case when $v_\epsilon= \zeta_\epsilon$. Recalling the terms from Lemma \ref{MP-geo} and Remark \ref{multrmk}, we define $A = \max\limits_{t \in [0,1]}I_{0,\epsilon}(tTe_1)$ then
\[A \geq \max_{t \in [0,1]} I_{\la,\epsilon}(tTe_1) \geq\inf_{\gamma\in\Gamma}\max_{t \in [0,1]}I_{\la,\epsilon}(\gamma (t)) = I_{\la,\epsilon}(\zeta_\epsilon)\geq \rho>0>I_{\la,\epsilon}(\nu_\epsilon).\]
Therefore
\begin{equation}\label{ap1}
\frac{1}{2}\int_{\Om}|\nabla\zeta_\epsilon|^2\,dx+\frac{1}{2}\iint_{\mathbb{R}^{2n}}\frac{|\zeta_{\epsilon}(x)-\zeta_{\epsilon}(y)|^2}{|x-y|^{n+2s}}\, dx dy-{\la}\int_\Om \frac{(\zeta_\epsilon +\epsilon)^{1-\gamma}-\epsilon^{1-\gamma}}{1-\gamma}~dx -\frac{1}{q+1}\int_\Om \zeta_\epsilon^{q+1}~dx \leq A.
\end{equation}
Choosing $\phi=-\frac{\zeta_\epsilon}{2}$ as a test function in \eqref{apx} we obtain
\begin{equation}\label{ap2}
-\frac{1}{q+1}\int_{\Om}|\nabla\zeta_\epsilon|^2\,dx-\frac{1}{q+1}\iint_{\mathbb{R}^{2n}}\frac{|\zeta_{\epsilon}(x)-\zeta_{\epsilon}(y)|^2}{|x-y|^{n+2s}}\, dx dy+\frac{\la}{q+1}\int_{\Om}\frac{\zeta_\epsilon}{(\zeta_{\epsilon}+\epsilon)^{\gamma}}\,dx+\frac{1}{q+1}\int_{\Om}\zeta_{\epsilon}^{q+1}\,dx=0.
\end{equation}
Adding \eqref{ap1} and \eqref{ap2} we have
\begin{align*}
\left(\frac{1}{2}-\frac{1}{q+1}\right)\|\zeta_{\epsilon}\|^2
 &\leq {\la}\int_\Om \frac{(\zeta_\epsilon +\epsilon)^{1-\gamma}-\epsilon^{1-\gamma}}{1-\gamma}~dx -\frac{\la}{q+1}\int_{\Om}\frac{\zeta_\epsilon}{(\zeta_\epsilon +\epsilon)^{\gamma}}\,dx+A\\
 & \leq C\int_\Om {\zeta_\epsilon} ^{1-\gamma}+A\leq C\|\zeta_\epsilon\|^{1-\gamma}+A,
\end{align*}
for some positive constant $C$ being independent of $\epsilon$, where we have used H\"older's inequality and Lemma \ref{emb}. Thus, since $q>1$, the sequence $\{\zeta_\epsilon\}$ is uniformly bounded in $H_0^{1}(\Omega)$ with respect to $\epsilon$. This completes the proof.
\end{proof}

\subsection{Proof of Theorem \ref{pu2thm}}
By Lemma \ref{non-negative} and Lemma \ref{apriori}, up to a subsequence, $\zeta_\epsilon \rightharpoonup \zeta_0$ and $\nu_\epsilon \rightharpoonup \nu_0$ weakly in $H_0^{1}(\Omega)$ as $\epsilon \to 0^+$, for some nonnegative $\zeta_0,\nu_0\in H_0^{1}(\Omega)$.\\
\textbf{Step $1$.} Let $v_0=\zeta_0$ or $\nu_0$. Here, we prove that $v_0\in H_0^{1}(\Omega)$ is a weak solution of the problem \eqref{psp}. Indeed, for any $\epsilon\in(0,1)$ and $t\geq 0$, we notice that
$$
{\la}{(t+\epsilon)^{-\gamma}}+t^q\geq {\la}{(t+1)^{-\gamma}}+t^q\geq \text{min}\left\{1,\frac{\la}{2}\right\}:=C>0, \text{ say}.
$$
Therefore, recalling that $v_\epsilon=\zeta_\epsilon$ or $\nu_\epsilon$, we have
\begin{equation}\label{eps}
-\Delta v_\epsilon+(-\Delta)^s v_{\epsilon}={\la}{(v_\epsilon+\epsilon)^{-\gamma}}+v_\epsilon^q\geq C>0.
\end{equation}
Using \cite[Lemma 3.1]{GU} (see also \cite[Lemma 3.1]{Aro-Rad}), we get the existence of $\xi\in H_0^{1}(\Omega)\cap L^{\infty}(\Omega)$ satisfying
$$
-\Delta\xi+(-\Delta)^s \xi=C\text{ in }\Omega,\,\,\xi>0\text{ in }\Omega,\,\,\xi=0\text{ in }\mathbb{R}^n\setminus\Omega
$$
such that for every $\omega\Subset\Om$, there exists a constant $c(\omega)>0$ satisfying $\xi\geq c(\omega)>0$ in $\Om$. 
Then, for every nonnegative $\phi\in H_0^{1}(\Omega)$, we have
\begin{align*}
&\int_{\Om}\nabla v_\epsilon\nabla\phi\,dx+\iint_{\mathbb{R}^{2n}}\frac{(v_{\epsilon}(x)-v_{\epsilon}(y))(\phi(x)-\phi(y))}{|x-y|^{n+2s}}\,dx dy=\int_{\Om}\Big({\la}{(v_\epsilon+\epsilon)^{-\gamma}}+v_\epsilon^q\Big)\phi\,dx\geq\int_{\Om}C\phi\,dx\\
&=\int_{\Om}\nabla \xi\nabla\phi\,dx+\iint_{\mathbb{R}^{2n}}\frac{(\xi(x)-\xi(y))(\phi(x)-\phi(y))}{|x-y|^{n+2s}}\,dx dy.
\end{align*}
Testing with $\phi=(\xi-v_\epsilon)^+$ in the above estimate, we obtain
$$
\int_{\Om}|\nabla(\xi-v_\epsilon)^+|^2\,dx+\iint_{\mathbb{R}^{2n}}\frac{(\xi(x)-\xi(y)-(v_{\epsilon}(x)-v_{\epsilon}(y))((\xi-v_{\epsilon})^+(x)-(\xi-v_{\epsilon})^+(y))}{|x-y|^{n+2s}}\,dx dy\leq 0.
$$
Following the same arguments as in the proof of \cite[Lemma 9]{LL}, the double integral in the above estimate become nonnegative. Hence, using this fact in the above inequality gives $v_\epsilon\geq \xi$ in $\Om$. Hence there exists a constant $c(\omega)>0$ (independent of $\epsilon$) such that
\begin{equation}\label{uniform}
v_\epsilon\geq c(\omega)>0,\text{ for every }\omega\Subset\Om. 
\end{equation}
Using Lemma \ref{apriori} and the fact \eqref{uniform} along with the hypothesis on $q$, we can pass to the limit in \eqref{eps} to obtain
\begin{equation*}
\begin{split}
&\int_{\Om}\nabla v_0\nabla \phi\,dx+\iint_{\mathbb{R}^{2n}}\frac{(v_0(x)-v_0(y)-(v_{\epsilon}(x)-v_{\epsilon}(y))((\xi-v_{\epsilon})^+(x)-(\xi-v_{\epsilon})^+(y))}{|x-y|^{n+2s}}\,dx dy\\
&=\la\int_{\Om}{\phi}{v_0^{-\gamma}}(x)\,dx+\int_{\Om}v_0^{q}\phi\,dx,
\end{split}
\end{equation*}
for every $\phi\in C_c^{1}(\Omega)$. Hence the claim follows.\\

\textbf{Step $2$.} Now we establish that $\zeta_0\neq \nu_0$. Choosing $\phi=v_\epsilon\in H_0^{1}(\Om)$ as a test function in \eqref{apx}, we get
$$
\int_{\Om}|\nabla v_\epsilon|^{2}\,dx+\iint_{\mathbb{R}^{2n}}\frac{|v_{\epsilon}(x)-v_{\epsilon}(y)|^2}{|x-y|^{n+2s}}\,dx dy=\la\int_{\Om}{v_\epsilon}{(v_\epsilon +\epsilon)^{-\gamma}}\,dx+\int_{\Om}v_\epsilon^{q+1}\,dx.
$$
Since $q+1<l$, using Lemma \ref{emb} we obtain
\begin{equation}\label{nl}
\lim\limits_{\epsilon\to 0^+}\int_{\Om}(v_\epsilon)^{q+1}\,dx=\int_{\Om}v_0^{q+1}\,dx.
\end{equation}
Moreover, since
$$
0\leq {v_\epsilon}{(v_\epsilon +\epsilon)^{-\gamma}}\leq v_\epsilon^{1-\gamma},
$$
using Vitali's convergence theorem, it follows that
$$
\la\lim\limits_{\epsilon\to 0^+}\int_{\Om}{v_\epsilon}{(v_\epsilon +\epsilon)^{-\gamma}}\,dx=\la\int_{\Om}v_0^{1-\gamma}\,dx.
$$
Therefore, we obtain
\begin{equation}\label{1}
\lim_{\epsilon\to 0^+}\Big(\int_{\Om}|\nabla v_\epsilon|^{2}\,dx+\iint_{\mathbb{R}^{2n}}\frac{|v_{\epsilon}(x)-v_{\epsilon}(y)|^2}{|x-y|^{n+2s}}\,dx dy\Big)=\la\int_{\Om}v_0^{1-\gamma}\,dx+\int_{\Om}v_0^{q+1}\,dx.
\end{equation}
By Remark \ref{tstrmk}, choosing $\phi=v_0$ as a test function in \eqref{psp} we get
\begin{equation}\label{2}
\int_{\Om}|\nabla v_0|^{2}\,dx+\iint_{\mathbb{R}^{2n}}\frac{|v_0(x)-v_0(y)|^2}{|x-y|^{n+2s}}\,dx dy=\la\int_{\Om}v_0^{1-\gamma}\,dx+\int_{\Om}v_0^{q+1}\,dx.
\end{equation}
Hence from \eqref{1} and \eqref{2}, we obtain
\begin{equation}\label{3}
\lim\limits_{\epsilon\to 0^+}\Big(\int_{\Om}|\nabla v_\epsilon|^{2}\,dx+\iint_{\mathbb{R}^{2n}}\frac{|v_{\epsilon}(x)-v_{\epsilon}(y)|^2}{|x-y|^{n+2s}}\,dx dy\Big)=\int_{\Om}|\nabla v_0|^{2}\,dx+\iint_{\mathbb{R}^{2n}}\frac{|v_0(x)-v_0(y)|^2}{|x-y|^{n+2s}}\,dx dy.
\end{equation}
Using Vitali's convergence theorem, we have
\begin{equation}\label{4}
\lim\limits_{\epsilon\to 0^+}\int_{\Om}[(v_\epsilon +\epsilon)^{1-\gamma}-\epsilon^{1-\gamma}]\,dx=\int_{\Om}v_0^{1-\gamma}\,dx.
\end{equation}
From \eqref{nl}, \eqref{3} and \eqref{4}, we have
$
\lim\limits_{\epsilon\to 0^+}I_{\la,\epsilon}(v_\epsilon)=I_{\la}(v_0),
$
which along with \eqref{limit-pass} gives $\zeta_0\neq \nu_0$.


\noindent {\textsf{Prashanta Garain\\Department of Mathematical Sciences\\
Indian Institute of Science Education and Research Berhampur\\
Berhampur, Odisha 760010, India\\
Department of Mathematics,\\
Indian Institute of Technology Indore,\\
Khandwa Road, Simrol, Indore 453552, India}\\ 
\textsf{e-mail}: pgarain92@gmail.com\\

\end{document}